\documentclass{amsart}
\usepackage{amsmath,graphicx,amssymb,amsthm}
\usepackage{color}
\def\A{\mathcal{A}}
\def\B{\mathcal B}
\def\F{\mathcal F}

\def\K{\mathcal K}

\def\B{\mathcal B}
\def\C{\mathcal C}
\def\H{\mathcal H}
\def\K{\mathcal K}

\def\P{\mathcal P}

\def\X{\mathcal X}

\def\amslatex{$\mathcal{A}\kern-.1667em\lower.5ex\hbox{$\mathcal{M}$}\kern-.125em\mathcal{S}$-\LaTeX}

\newtheorem{set}{set}[section]
\newtheorem{Corollary}[set]{Corollary}
\newtheorem{Definition}[set]{Definition}

\newtheorem{Lemma}[set]{Lemma}

\newtheorem{Proposition}[set]{Proposition}

\newtheorem{Theorem}[set]{Theorem}
\newcommand{\define}{\mathrel{\hbox{$\equiv$\hskip -.90em \lower .47ex \hbox{$\leftharpoondown$}}}}
\newcommand{\enifed}{\mathrel{\hbox{$\equiv$\hskip -.90em \lower .47ex \hbox{$\rightharpoondown$}}}}

\begin{document}
\title {\bf  Two-faced Families of Non-commutative Random Variables Having Bi-free Infinitely Divisible Distributions}
\author{ Mingchu Gao\address{Department of Mathematics,
Louisiana College,
Pineville, LA 71359,
USA,
Email: mingchu.gao@lacollege.edu
} }
\date{}
\maketitle
\begin{abstract}
We study two-faced families of non-commutative random variables having bi-free (additive) infinitely divisible distributions. We prove a limit theorem of the sums of bi-free two-faced families of random variables within a triangular array. As a corollary of our limit theorem, we get Voiculescu's bi-free central limit theorem. Using the full Fock space operator model,  we show that a two-faced pair of random variables has a bi-free (additive) infinitely divisible distribution if and only if its distribution is the limit distribution in our limit theorem.  Finally, we characterize the bi-free (additive) infinite divisibility of the distribution of a two-faced pair of random variables in terms of bi-free Levy processes.
\end{abstract}


\section*{Introduction}

On a free product of Hilbert spaces with a specified unit vector
there are two actions  of the operators of the initial space,
corresponding to a left and a right tensorial factorization
respectively (\cite{VDN}). In free probability, free random
variables can be modeled as left {\sl or} right actions on free
products of spaces. When  considering the algebra of left actions
{\sl and} the algebra of right actions on the free product space
simultaneously, Voiculescu \cite{Vo1} discovered  a new phenomenon
of  freely independent family of two-faced families of random
variables, which is called {\sl bi-free independence}. This
``two-faced" extension of free probability is called {\sl bi-free
probability} (or free probability for two-faced pairs) introduced in
\cite{Vo1} and \cite{Vo2} recently.

Bi-free independence is a more general notion than free independence. Classical independence and free independence can be treated as special cases of bi-freeness (\cite{PS}).  Voiculescu \cite{Vo1} demonstrated that many results in free probability such as existence of free cumulants and the free central limit theorem, have direct analogues in the bi-free setting. Voiculescu \cite{Vo1} showed the existence of bi-free cumulant polynomials, but did not give explicit formulas for the polynomials. Mastnak and Nica \cite{MN} introduced combinatorial objects called {\sl bi-non-crossing partitions}, associated a family of {\sl (l,r)-cumulants} with the combinatorial objects, and conjectured that bi-freeness was equivalent to the vanishing of these mixed cumulants. This conjecture was later proved by Charlesworth, Nelson, and Skoufranis in \cite{CNS1}. The same authors developed a theory of bi-freeness in an amalgamated setting in \cite{CNS2}.

In classical probability, infinitely divisible distributions appear
as a broad generalization of the central limit theorem and the Poisson
limit theorem: the limit distribution of the sums of i.i.d. random
variables within a triangular array.  Meanwhile,  infinitely
divisible distributions are closely  related to Levy processes, a
very important research area in probability theory (\cite{PSa}).    There has been a
well-developed theory on infinitely divisible distributions in free
probability: the limit distributions of free random variables within
a triangular array,  infinitely divisible distributions with respect
to the additive free convolution, and free Levy processes have a
perfect relation similar to those in classical probability
(\cite{NS}).  In bi-free probability,  Voiculescu \cite{Vo1} defined centered
bi-free Gaussian distributions and proved
an algebraic bi-free central limit theorem,  a
bi-free version of semicircle
distributions and the free central limit theorem, respectively  (7.2 and 7.3 in \cite{Vo1}). In \cite{Vo2}, Voiculescu studied a special kind of
two-faced families of random variables
where algebraic relations between left and right random variables
ensure that all moments can be computed from the ``two-band''
moments $\varphi(LR)$, where $L$ and $R$ are monomials in left and,
respectively, right random variables. A simplest example is
bi-partite systems where left and right variables commute with each
other.    In \cite{GHM}, the authors studied infinitely divisible
distributions in the bi-free probability setting for bi-partite
systems.   They derived a  bi-free analogue of the Levy-Hincin
formula for infinitely divisible distributions with respect to the
additive bi-free convolution,  constructed bi-free Levy processes
corresponding to bi-free infinitely divisible distributions, and
proved a bi-free limit theorem.
In this paper, we study infinitely divisible distributions with
respect to the additive bi-free convolution for general two-faced
pairs of non-commutative random variables,
i.e., we do not assume that the left and right random variables of the
two-faced pair commute with each other.

Besides this introduction, this paper contains 4 sections. In section 1  we review  the basic knowledge on bi-free probability used in sequel. In Section 2, we prove  limit theorems (Theorem 2.3 and Corollary 2.4) in bi-free probability, generalizing the limit theorem in free probability (Theorem 13.1 in \cite{NS}) and the limit theorem in bi-free probability for bi-partite systems (Theorem 3.1 in \cite{GHM}) to the general case of two-faced families of random variables. As a corollary of our limit theorems, we get Voiculescu's bi-free central limit theorem (Proposition 2.6). Section 3 is devoted to studying  two-faced pairs of random variables having bi-free infinitely divisible distributions. We show that a two-faced pair $(a,b)$ of self adjoint operators in a $C^*$-probability apace $(\A,\varphi)$ has a bi-free infinitely divisible distribution if and only if its distribution is the limit distribution of the sums of bi-free independent two-faced pairs of random variables within a triangular array (Theorem 3.7). Finally, in Section 4, we define  bi-free Levy processes (Definition 4.1). Our definition is very similar  to Definition 4.1 in \cite{GHM}. But we do not assume that the system is bi-partite. We find that a bi-free infinitely divisible distribution is exactly the distribution of $a_1=(a_{l,1}, a_{r,1})$ of a bi-free Levy process $\{a_t=(a_{l,t}, a_{r,t}):t\ge 0\}$ (Theorem 4.2), generalizing Theorem 4.2 in \cite{GHM} to the general case of two-faced pairs of random variables.

We refer the reader to \cite{VDN} and \cite{NS} for  free probability, to \cite{Vo1}, \cite{Vo2}, \cite{MN}, \cite{CNS1}, and \cite{CNS2} for bi-free probability, and to \cite{KR} for operator algebras.

{\bf Acknowledgement} The author is grateful to the referee(s) for carefully reviewing the paper, providing many valuable suggestions,  and finding out typos in the paper.  The author would like to thank Dr. Paul Skoufranis of Texas A. and M. University and Prof. Guimei An of Nankai University, China,  for pointing out mistakes and typos in the initial version of this paper.

\section{Preliminaries}

In this section, we review some basic concepts and results in bi-free probability used in sequel.

 {\bf Free Products of vector spaces} Let $\mathcal{X}$ be a vector space with a  vector $\xi$, and a subspace $\mathcal{X}_0$ with co-dimension 1 such that $\mathcal{X}=\mathbb{C}\xi\oplus\mathcal{X}_0$. Let $L(\X)$ be the space of all linear operators on $\X$. We use $(\mathcal{X}, \mathcal{X}_0, \xi)$ to denote vector space $\mathcal{X}$ with the above decomposition property.     Let $(\mathcal{X}_i, \mathcal{X}_{i,0}, \xi_i), i\in I$, be a family of vector spaces. The free product $(\mathcal{X}, \mathcal{X}_0,\xi)=*_{i\in I}(\mathcal{X}_i, \mathcal{X}_{i,0}, \xi_i)$ is defined as
 $$\mathcal{X}=\mathbb{C}\xi\oplus\mathcal{X}_0, \mathcal{X}_0=\oplus_{n\ge 1}(\oplus_{i_1\ne i_2\ne \cdots\ne i_n}\mathcal{X}_{i_1, 0}\otimes\mathcal{X}_{i_2,0}\otimes\cdots\otimes \mathcal{X}_{i_n,0}).$$
 We can define a linear functional $\phi: \mathcal{X}\rightarrow \mathbb{C}$ by $\phi(\xi)=1, \ker(\phi)=\mathcal{X}_0$, then a linear functional $\varphi_\xi:L(\X)\rightarrow \mathbb{C}$, $\varphi_\xi(T)=\phi( T\xi), \forall T\in L(\X)$. The pair $(L(\X), \varphi_\xi)$ is a non-commutative probability space.

  {\bf A Tensor Product Factorization of Free Product Vector Spaces} Let $\mathcal{X}=\mathcal{X}_0\oplus \mathbb{C}\xi$ be the free product of a family $\{(\mathcal{X}_i, \mathcal{X}_{0,i},\xi_i):i\in I\}$ of vector spaces.
 For $i\in I$, define
 $$\mathcal{X}(l,i)=\mathbb{C}\xi \oplus\bigoplus_{n\ge 1}(\bigoplus_{i_1\ne i_2\ne\cdots\ne i_n, i_1\ne i}\mathcal{X}_{0,i_1}\otimes\mathcal{X}_{0,i_2}\oplus\cdots\oplus\mathcal{X}_{0,i_n}),$$ and
 $$\mathcal{X}(r,i)=\mathbb{C}\xi \oplus\bigoplus_{n\ge 1}(\bigoplus_{i_1\ne i_2\ne\cdots\ne i_n, i_n\ne i}\mathcal{X}_{0,i_1}\otimes\mathcal{X}_{0,i_2}\oplus\cdots\oplus\mathcal{X}_{0,i_n}).$$
 Define a unitary operator $V_i: \mathcal{X}_i\otimes \mathcal{X}(l,i)\rightarrow \mathcal{X}$ by
 $$V_i: \xi_i\otimes\xi\mapsto \xi, \mathcal{X}_{0,i}\otimes\xi\mapsto \mathcal{X}_{0,i},
  \xi_i\otimes (\mathcal{X}_{0,i_1}\otimes\mathcal{X}_{0,i_2}\oplus\cdots\oplus\mathcal{X}_{0,i_n})\mapsto\mathcal{X}_{0,i_1}\otimes\mathcal{X}_{0,i_2}\oplus\cdots\oplus\mathcal{X}_{0,i_n}, $$
 $$\mathcal{X}_{0,i}\otimes (\mathcal{X}_{0,i_1}\otimes\mathcal{X}_{0,i_2}\oplus\cdots\oplus\mathcal{X}_{0,i_n})\mapsto\mathcal{X}_{0,i}\otimes \mathcal{X}_{0,i_1}\otimes\mathcal{X}_{0,i_2}\oplus\cdots\oplus\mathcal{X}_{0,i_n}.$$
 We can define  $W_i:\mathcal{X}(r,i)\otimes \mathcal{X}_i\rightarrow \mathcal{X}$ similarly.

An operator $T\in L(\mathcal{X}_i)$ can act on  the free product space $(\mathcal{X}, \mathcal{X}_0,\xi)$ from leftmost $\lambda_i(T)$, or from rightmost $\rho_i(T)$:
  $$\lambda_i(T)=V_i(T\otimes I)V_i^*, \rho_i(T)=W_i(I\otimes T)W_i^*.$$

 {\bf Bi-free independence}.  A two-faced family of random variables $$(\widehat{b}, \widehat{c}):=((b_i)_{i\in I}, (c_j)_{j\in J})$$ is an ordered pair of  two families of elements in a non-commutative probability space $(\A,\varphi)$. For a two-faced family $(\widehat{b}, \widehat{c})=((b_i)_{i\in I}, (c_j)_{j\in J})$ of random variables in $\A$, its distribution $\mu_{\widehat{b},\widehat{c}}:\mathbb{C}\langle X_i,Y_j|i\in I,j\in J\rangle\rightarrow \mathbb{C}$ is defined as
 $$\mu_{\widehat{b}, \widehat{c}}(P((X_i)_{i\in I}, (Y_j)_{j\in J}))=\varphi (P((b_i)_{i\in I}, (c_j)_{j\in J})), \forall P\in \mathbb{C}\langle X_i,Y_j|i\in I,j\in J\rangle. $$
 A pair of faces (or face-pair) in a non-commutative probability space $(\A,\varphi)$ is an ordered pair $(\B, \C)$ of two unital subalgebras of $\A$. Let $\pi=((\B_k, \C_k))_{k\in K}$ be a family of pairs of faces in $(\A,\varphi)$. The {\sl joint distribution} of $\pi$ is the linear functional $\mu_\pi: *_{k\in K}\B_k*\C_k\rightarrow \mathbb{C}$ defined by $\mu_\pi=\varphi \circ \alpha$,  where $\alpha:*_{k\in K}\B_k*\C_k\rightarrow \A $ is the homomorphism such that $$\alpha|_{\B_k}(x)=x, \forall x\in \B_k, \alpha|_{\C_k}(x)=x, \forall x\in \C_k. $$

 Let $z'=((b'_i)_{i\in I}, (c'_j)_{j\in J})$ and $z''=((b''_i)_{i\in I}, (c''_j)_{j\in J})$ be two two-faced families in a non-commutative probability space $(\A,\varphi)$. We say that $z'$ and $z''$ are {\sl bi-free} if there exist a free product $(\mathcal{X}, p,\xi)=(\mathcal{X}', p',\xi')*(\mathcal{X}'', p'',\xi'')$ of vector spaces and homomorphisms $$l^\varepsilon: \mathbb{C}\langle b_i^\varepsilon:i\in I\rangle\rightarrow \mathcal{L}(\mathcal{X}^\varepsilon), r^\varepsilon: \mathbb{C}\langle c_j^\varepsilon:j\in J\rangle\rightarrow L(\mathcal{X}^\varepsilon),\varepsilon \in \{',''\},$$ such that $T^\varepsilon:=(\lambda^\varepsilon\circ l^\varepsilon(b_i^\varepsilon)_{i\in I}, \rho^\varepsilon\circ r^\varepsilon (c^{\varepsilon}_j)_{j\in J})$ with $\varepsilon \in \{',''\}$ have a joint distribution in $(L(\mathcal{X}), \varphi_\xi)$ same as that of $z'$ and $z''$ in $(\A,\varphi)$.

 {\bf An example of bi-free two-faced families of random variables}. Let $\H=\oplus_{i\in I}\H_i$ be the direct sum of  complex Hilbert spaces, and $\mathcal{F}(\H)=\mathbb{C}\Omega\oplus \bigoplus_{n\ge 1}\H^{\otimes n}$ be the full Fock space. Let $$\tau_\H:B(\F(\H))\rightarrow \mathbb{C}, \tau_\H(T)=\langle T\Omega, \Omega\rangle, \forall T\in B(\F(\H)),$$ be the vector state on $B(\F(\H))$ corresponding to the vacuum vector $\Omega\in \F(\H)$. Then $(B(\F(\H)), \tau_\H)$ is a $C^*$-probability space.

 For $f\in \H, T\in B(\H), \xi=\xi_1\otimes\cdots\otimes \xi_n\in \underbrace{\H\otimes\cdots\otimes\H}_{n \text{times}}$, define
 $$l(f)\Omega=f, l(f)\xi=f\otimes\xi, r(f)\Omega=f, r(f)\xi=\xi\otimes f, $$
 $$\Lambda_l(T)\Omega=0, \Lambda_l(T)\xi=(T\xi_1)\otimes \xi_2\otimes \cdots\otimes \xi_n, \Lambda_r(T)\Omega=0, \Lambda_r(T)\xi=\xi_1\otimes\cdots\otimes \xi_{n-1}\otimes(T\xi_n).$$

\begin{Proposition}[Remark 3.5 in \cite{GHM}] Let $\B_i$ and $\C_i$ be the $C^*$-algebras generated by $\{l(f):f\in \H_i\}\cup \{\Lambda_l(T): T\in B(\H), T\H_i\subset \H_i, T|_{\H\ominus\H_i}=0\}$ and $\{r(f):f\in \H_i\}\cup \{\Lambda_r(T): T\in B(\H), T\H_i\subset \H_i, T|_{\H\ominus\H_i}=0\}$, respectively. Then $\{(\B_i, \C_i):i\in I\}$ is bi-free in $(B(\F(\H)), \tau_\H)$.
\end{Proposition}
 {\bf Bi-free cumulants.}  Let $\chi=(h_1, h_2, \cdots, h_n)\in \{l,r\}^n$. Let's record explicitly where are the occurrences of $l$ and $r$ in $\chi$.
 $$\{m:1\le m\le n, h_m=l\}=\{m_l(1)<m_l(2)<\cdots, <m_l(u)\},$$
 $$ \{m:1\le m\le n, h_m=r\}=\{m_r(1)<m_r(2)<\cdots, <m_r(n-u)\}.$$
 Define a permutation $s_\chi:\{1, 2, \cdots, n\}\rightarrow \{1, 2, \cdots, n\}$, $s_\chi(i)=m_l(i)$, if $1\le i\le u$; $s_\chi(u+i)=m_r(n-u+1-i)$, if $1\le i\le n-u$.

 For a subset $V=\{i_1, i_2, \cdots, i_k\}$ of the set $[n]:=\{1, 2, \cdots, n\}$, $a_1\cdots, a_n\in \A$, define $$\varphi_V(a_1, \cdots, a_n)=\varphi(a_{i_1}a_{i_2}\cdots a_{i_k}).$$  Let $\P(n)$ be the set of all partitions of $[n]$. For a partition $\pi=\{V_1, V_2, \cdots, V_d\}\in \P(n)$, we define $$\varphi_\pi(a_1, \cdots, a_n):=\prod_{V\in \pi}\varphi_V(a_1, \cdots, a_n).$$  Define $\mathcal{P}^\chi(n)=\{s_\chi\circ \pi:\pi\in NC(n)\}$, where $NC(n)$ is the set of all non-crossing partitions of $[n]$ (Lecture 9 in \cite{NS}).  Let $(\A,\varphi)$ be a non-commutative probability space. The bi-free cumulants $(\kappa_\chi:\A^n\rightarrow \mathbb{C})_{n\ge 1, \chi\in \{\l,r\}^n}$ of $(\A,\varphi)$ are defined by
 $$\kappa_\chi (a_1, \cdots, a_n)=\sum_{\pi\in \mathcal{P}^{(\chi)}(n)}\varphi_{\pi}(a_1, \cdots, a_n)\mu_n(s^{-1}_\chi\circ\pi, 1_n)  \eqno (1.1)$$ for $n\ge 1, \chi \in \{l,r\}^n, a_1, \cdots, a_n\in A$, where $\mu_n$ is the Mobius function on $NC(n)$ (Lecture 10 in \cite{NS}). For a subset $V=\{i_1, i_2, \cdots, i_k\}\subseteq \{1, 2, \cdots, n\}$, let $\chi_V$ be the restriction of $\chi$ on $V$. We define $\kappa_{\chi, V}(a_1, a_2, \cdots, a_n)=\kappa_{\chi_V}(a_{i_1}, a_{i_2}, \cdots, a_{i_k})$.
 For a partition $\pi=\{V_1, V_2, \cdots, V_k\} \in \P^{(\chi)}$, we define $\kappa_{\chi, \pi}=\prod_{V\in \pi}\kappa_{\chi, V}(a_1, a_2, \cdots, a_n)$.  Then the bi-free cumulant appeared in $(1.1)$ is $\kappa_{\chi, 1_n}(a_1, \cdots, a_n)$. The bi-free cumulants are determined by the equation $$\varphi (a_1 a_2 \cdots a_n)=\sum_{\pi\in \P^{(\chi)}(n)}\kappa_{\chi,\pi}(a_1, a_2, \cdots, a_n), \forall a_1, \cdots, a_n\in \A,   \eqno (1.2)$$ for a $\chi:\{1, 2, \cdots, n\}\rightarrow \{l,r\}^n$.

 Charlesworth, Nelson, and Skoufranis \cite{CNS1} proved that two two-faced families
  $$z'=((z'_i)_{i\in I}, (z'_j)_{j\in J}), z''=((z''_i)_{i\in I}, (z''_j)_{j\in J})$$ in a non-commutative probability space $(\A,\varphi)$ are bi-free if and only if
 $$\kappa_\chi(z_{\alpha(1)}^{\epsilon_1},z_{\alpha(2)}^{\epsilon_2}, \cdots, z_{\alpha(n)}^{\epsilon_n})=0, \eqno (1.3)$$ whenever $\alpha:\{1,2,\cdots, n\}\rightarrow I\bigsqcup J$, $\chi:\{1, 2, \cdots, n\}\rightarrow \{l,r\}$ such that $\alpha^{-1}(I)=\chi ^{-1}(\{l\})$,   $\epsilon:[n]\rightarrow \{',''\}^n$ is not constant, and $n\ge 2$ (Theorem 4.3.1 in \cite{CNS1}).

\section{ Bi-free limit theorems}

Our goal, in this section,  is to prove  a bi-free limit theorem, an analogue of Theorem 13.1 in \cite{NS} in bi-free probability.

\begin{Lemma} Let $(\A, \varphi)$ be a non-commutative probability space, and for each $N\in \mathbb{N}$, $$\{(a_{l, N,1}, a_{r,N,1}), (a_{l,N,2}, a_{r,N,2}), \cdots, (a_{l, N,N}, a_{r,N,N})\}$$ be a sequence  of two-faced pairs of random variables in $\A$. An index tuple $\{(i(1), i(2), \cdots, i(n)): i(j)=1, 2, \cdots, N, j=1,2, \cdots, n\}$ corresponds to a partition $\pi$ of $\{1, 2, \cdots, n\}$ if $$p\sim_\pi q\Leftrightarrow i(p)=i(q), p, q=1, 2, \cdots, n. $$  The following  statements are equivalent.
 \begin{enumerate}
 \item For  all $n\in \mathbb{N}$, all $\pi\in \P(n)$, and all $\chi:\{1, 2, \cdots, n\}\rightarrow \{l,r\}$, limits $$\lim_{N\rightarrow \infty}\sum_{\sigma\le \pi, \sigma\in \P^{(\chi)}}N^{|\pi|}\kappa_{\chi,\sigma}(a_{\chi(1), N,i(1)}, a_{\chi(2), N,i(2)}, \cdots, a_{\chi(n), N,i(n)})$$ exist, where $\{i(1), i(2), \cdots, i(n)\}$ is an index tuple corresponding to partition $\sigma$.
\item The limits $$lim_{N\rightarrow \infty}N^{|\pi|}\kappa_{\chi,\sigma}(a_{\chi(1), N,i(1)}, a_{\chi(1), N, i(2)}, \cdots, a_{\chi(n), N,i(n)})$$ exist, for all $\{i(1), i(2), \cdots, i(n)\}$ correspond to partition $\sigma$, all $\chi:\{1, 2, \cdots, n\}\rightarrow \{l,r\}$ all $\sigma\in \P^{(\chi)}(n),  \sigma\le \pi$, all $\pi \in \P(n)$, and all $n\in \mathbb{N}$.
\item For each $n\in \mathbb{N}$, $\chi:\{1, 2, \cdots, n\}\rightarrow \{l,r\}$, limit $$\lim_{N\rightarrow \infty}N\kappa_{\chi, 1_n}(a_{\chi(1), N,i}, a_{\chi(2), N, i}, \cdots, a_{\chi(n), N,i})$$ exists, where $i=1, 2, \cdots, N$.
\end{enumerate}
\end{Lemma}
\begin{proof} $(2)\Rightarrow (1)$ is obvious, since the sum in (1) is a sum of finite terms.

$(1)\Rightarrow (2)$. For every $n\in \mathbb{N}$,  $\pi=1_n$, $|\pi|=1$, by (1.2), we have
\begin{align*}
&\lim_{N\rightarrow \infty}\sum_{\sigma\le \pi, \sigma\in \P^{(\chi)}}N^{|\pi|}\kappa_{\chi,\sigma}(a_{\chi(1), N,i(1)}, a_{\chi(2), N,i(2)}, \cdots, a_{\chi(n), N,i(n)})\\
=&\lim_{N\rightarrow \infty}N\sum_{\sigma\in \P^{(\chi)}}\kappa_{\chi,\sigma}(a_{\chi(1), N,i(1)}, a_{\chi(2), N,i(2)}, \cdots, a_{\chi(n), N,i(n)})\\
=&\lim_{N\rightarrow \infty}N\varphi(a_{\chi(1), N,i}a_{\chi(2), N, i}\cdots a_{\chi(n), N, i})
 \end{align*}
 exists, $i=1, 2, \cdots, N$.
 It implies that for  a partition $\pi=\{V_1, V_2, \cdots, V_d\}$,
 \begin{multline*}
 \lim_{N\rightarrow \infty}N^{|\pi|}\varphi_\pi(a_{\chi(1), N,i(1)}, a_{\chi(2), N, i(2)}, \cdots, a_{\chi(n), N,i(n)})\\
 =\lim_{N\rightarrow \infty}\prod_{j=1}^dN\varphi_{V_j}(a_{\chi(1), N,i}, a_{\chi(2), N, i}, \cdots, a_{\chi(n), N,i})
 \end{multline*} exists. From this, we have for $\pi\in \P^{(\chi)}(n)$, by (1.1),
 \begin{align*}
 &\lim_{N\rightarrow \infty}N^{|\pi|}\kappa_{\chi,\pi}(a_{\chi(1), N,i(1)}, a_{\chi(1), N, i(2)}, \cdots, a_{\chi(n), N,i(n)})\\
 =&\lim_{N\rightarrow \infty}\prod_{V\in \pi}N\kappa_{\chi,V}(a_{\chi(1), N,i(1)}, a_{\chi(1), N, i}, \cdots, a_{\chi(n), N,i})\\
 =&\lim_{N\rightarrow \infty}\prod_{V\in \pi}\sum_{\sigma\in \mathcal{P}^{(\chi|_V)}(|V|)}\frac{N}{N^{|\sigma|}}N^{|\sigma|}\varphi_{\sigma}((a_{\chi(1), N,i(1)}, a_{\chi(1), N, i(2)}, \cdots, a_{\chi(n), N,i(n)})|_V)\mu_{|V|}(s^{-1}_\chi(\sigma), 1_{|V|})\\
 =&\prod_{V\in \pi}\sum_{\sigma\in \mathcal{P}^{(\chi|_V)}(V)}(\lim_{N\rightarrow\infty}\frac{N}{N^{|\sigma|}}\\
 \times & \lim_{N\rightarrow\infty}N^{|\sigma|}\varphi_{\sigma}((a_{\chi(1), N,i(1)}, a_{\chi(1), N, i(2)}, \cdots, a_{\chi(n), N,i(n)})|_V)\mu_{|V|}(s^{-1}_\chi(\sigma), 1_{|V|}))
 \end{align*}
 exists. Therefore, for $\sigma\in \P^{(\chi)}, \sigma\le \pi, \pi\in \P(n), \chi:\{1, 2, \cdots, n\}\rightarrow \{l,r\}^n$, the limit
 \begin{multline*}
 \lim_{N\rightarrow \infty}N^{|\pi|}\kappa_{\chi,\sigma}(a_{\chi(1), N,i(1)}, a_{\chi(2), N,i(2)}, \cdots, a_{\chi(n), N,i(n)})\\
 =\lim_{N\rightarrow\infty}\frac{N^{|\pi|}}{N^{|\sigma|}}\lim_{N\rightarrow \infty}N^{|\sigma|}\kappa_{\chi,\sigma}(a_{\chi(1), N,i(1)}, a_{\chi(2), N,i(2)}, \cdots, a_{\chi(n), N,i(n)})
 \end{multline*} exists.

 $(2)\Rightarrow (3)$ is obvious. $(3)\Rightarrow (2)$ is also obvious, by the proof of $(1)\Rightarrow (2)$.
\end{proof}
The following lemma is a bi-free probability version of Lemma 13.2 in \cite{NS}. Using equations (1.1) and (1.2), we can get a proof as same as that of Lemma 13.2 in \cite{NS}.
\begin{Lemma} Let $(\A, \varphi)$ be a non-commutative probability space, and for each $N\in \mathbb{N}$, $$\{(a_{l, N,1}, a_{r,N,1}), (a_{l,N,2}, a_{r,N,2}), \cdots, (a_{l, N,N}, a_{r,N,N})\}$$ be a sequence  of two-faced pairs of random variables in $\A$.Then the following two statements are equivalent.
\begin{enumerate}
\item For each $n\in \mathbb{N}$, $\chi:\{1, 2, \cdots, n\}\rightarrow \{l,r\}$, limit $$\lim_{N\rightarrow \infty}N\kappa_{\chi, 1_n}(a_{\chi(1), N,i}, a_{\chi(2), N, i}, \cdots, a_{\chi(n), N,i})$$ exists, where $i=1, 2, \cdots, N $.
\item The limit $$\lim_{N\rightarrow \infty}N\varphi(a_{\chi(1), N,i} a_{\chi(1), N, i} \cdots a_{\chi(n), N, i})$$ exists, for every $i=1, 2, \cdots, N$, $\chi:\{1, 2, \cdots, n\}\rightarrow \{l,r\}$, and $n\in \mathbb{N}$.
\end{enumerate}
If the above limits exist, they are equal to one another, i. e.,
$$\lim_{N\rightarrow \infty}N\kappa_{\chi, 1_n}(a_{\chi(1), N,i}, a_{\chi(2), N, i}, \cdots, a_{\chi(n), N,i})=\lim_{N\rightarrow \infty}N\varphi(a_{\chi(1), N,i} a_{\chi(1), N, i} \cdots a_{\chi(n), N, i}).$$
\end{Lemma}
For each $N\in \mathbb{N}$, let $$\{a_{N,1}=(a_{l,N,1}, a_{r,N,1}), a_{N,2}=(a_{l,N,2}, a_{r,N,2}), \cdots, a_{N,N}=(a_{l,N,N}, a_{r,N,N})\}$$ be a bi-free family of $N$ identically distributed two-faced fairs of random variables in a non-commutative probability $(\A,\varphi)$. It follows that, for $n\in \mathbb{N}, \chi:\{1, 2, \cdots, n\}\rightarrow \{l,r\}$, the moments $$\varphi(a_{\chi(1), N,i}\cdots a_{\chi(n), N, i}), 1\le i\le N,$$ are independent of $i$. Let $S_{h, N}=\sum_{i=1}^Na_{h,N,i}$, for $h\in \{l.r\}$, and $S_N=(S_{l,N}, S_{r,N})$ be the two-faced pair of random variables in $(\A,\varphi)$. Then we have the following limit theorem.
\begin{Theorem} The sequence of two-faced pairs $\{S_N:N\ge 1\}$ converges in distribution to a two-faced pair $b=(b_l, b_r)$ in a non-commutative probability space $(\B,\phi)$, as $N\rightarrow \infty$, i. e.,  $$S_N\stackrel{distr}{\longrightarrow}(b_l,b_r), N\rightarrow \infty, \eqno (2.1)$$ if and only if for each $n\ge 1$, and $\chi:\{1, 2, \cdots, n\}\rightarrow \{l,r\}$, the limit
$$\lim_{N\rightarrow \infty}N\varphi(a_{\chi(1),N,1}a_{\chi(2),N,1}\cdots a_{\chi(n),N,1}) \eqno (2.2)$$ exists.

Furthermore, if the limits exist, then the joint distribution of the limit pair $b=(b_l,b_r)$ is determined in terms of bi-free cumulants by $$\kappa_\chi(b):=\kappa_{\chi, 1_n}(b_{\chi(1)}, b_{\chi(2)}, \cdots, b_{\chi(n)})=\lim_{N\rightarrow \infty}N\varphi(a_{\chi(1),N,1} a_{\chi(2),N,1}\cdots a_{\chi(n),N,1}),$$ for $\chi:\{1, 2, \cdots, n\}\rightarrow \{l,r\}$.
\end{Theorem}
\begin{proof}
We follow the idea in the proof of Theorem 13.1 in \cite{NS}. For an $n\ge 1$, $\chi:\{1, 2, \cdots, n\}\rightarrow \{l,r\}$, we have
\begin{align*}
\lim_{N\rightarrow \infty}\varphi_\chi(S_N):=&\lim_{N\rightarrow\infty}\varphi (S_{\chi(1),N}S_{\chi(2), N}\cdots S_{\chi(n),N})\\
=&\lim_{N\rightarrow\infty}\sum_{i(1), i(2), \cdots, i(n)=1}^N\varphi(a_{\chi(1),N,i(1)}a_{\chi(2),N,i(2)}\cdots a_{\chi(n),N,i(n)})\\
=&\lim_{N\rightarrow \infty}\sum_{\pi\in \P(n)}N(N-1)\cdots(N-|\pi|+1)\varphi(a_{\chi(1),N,i(1)}a_{\chi(2),N,i(2)}\cdots a_{\chi(n),N,i(n)})\\
=&\sum_{\pi\in \P(n)}(\lim_{N\rightarrow \infty}\frac{N(N-1)\cdots(N-|\pi|+1)}{N^{|\pi|}}\\
\times &\lim_{N\rightarrow \infty}N^{|\pi|}\varphi(a_{\chi(1),N,i(1)}a_{\chi(2),N,i(2)}\cdots a_{\chi(n),N,i(n)}))\\
=&\lim_{N\rightarrow \infty}\sum_{\pi\in \P(n)}N^{|\pi|}\varphi(a_{\chi(1),N,i(1)}a_{\chi(2),N,i(2)}\cdots a_{\chi(n),N,i(n)}),
\end{align*}
where $\{i(1), i(2), \cdots, i(n)\}$ is an index tuple corresponding to partition $\pi$. By the discussion in the proof of Theorem 13.1 in \cite{NS}, $\{S_N:N\ge 1\}$ converges in distribution if and only if $$\lim_{N\rightarrow \infty}N^{|\pi|}\varphi(a_{\chi(1),N,i(1)}a_{\chi(2),N,i(2)}\cdots a_{\chi(n),N,i(n)})$$ exists for all $\pi \in \P(n)$, where $\{i(1), i(2), \cdots, i(n)\}$ is an index tuple corresponding to partition $\pi$.  For a $\pi\in \P(n)$, by (1.2) and (1.3), we have
\begin{align*}
&\lim_{N\rightarrow \infty}N^{|\pi|}\varphi(a_{\chi(1),N,i(1)}a_{\chi(2),N,i(2)}\cdots a_{\chi(n),N,i(n)})\\
=&\lim_{N\rightarrow \infty}\sum_{\sigma\in \P^{(\chi)}}N^{|\pi|}\kappa_{\chi,\sigma}(a_{\chi(1),N,i(1)},a_{\chi(2),N,i(2)},\cdots, a_{\chi(n),N,i(n)})\\
=&\lim_{N\rightarrow \infty}\sum_{\sigma\in \P^{(\chi)}, \sigma\le \pi}N^{|\pi|}\kappa_{\chi,\sigma}(a_{\chi(1),N,i(1)},a_{\chi(2),N,i(2)},\cdots, a_{\chi(n),N,i(n)})
\end{align*}
By Lemmas 2.1 and 2.2, we get that $(2.1)$ is equivalent to $(2.2)$.

If the existence of the limits is assumed, then
\begin{align*}
&\lim_{N\rightarrow\infty}\varphi (S_{\chi(1),N}S_{\chi(2), N}\cdots S_{\chi(n),N})\\
=&\lim_{N\rightarrow \infty}\sum_{\pi\in \P(n)}N^{|\pi|}\varphi(a_{\chi(1),N,i(1)}a_{\chi(2),N,i(2)}\cdots a_{\chi(n),N,i(n)})\\
=&\lim_{N\rightarrow \infty}\sum_{\pi\in \P(n)}\sum_{\sigma\in \P^{(\chi)}, \sigma\le \pi}N^{|\pi|}\kappa_{\chi,\sigma}(a_{\chi(1),N,i(1)},a_{\chi(2),N,i(2)},\cdots, a_{\chi(n),N,i(n)})\\
=&\lim_{N\rightarrow \infty}\sum_{\pi\in \P^{\chi}(n)}N^{|\pi|}\kappa_{\chi,\pi}(a_{\chi(1),N,i(1)},a_{\chi(2),N,i(2)},\cdots, a_{\chi(n),N,i(n)})\\
=&\sum_{\pi\in \P^{\chi}(n)}\lim_{N\rightarrow \infty}N^{|\pi|}\kappa_{\chi,\pi}(a_{\chi(1),N,i(1)},a_{\chi(2),N,i(2)},\cdots, a_{\chi(n),N,i(n)})
\end{align*}
On the other hand, by equation $(2.1)$, for every $n\in \mathbb{N}, \chi:\{1, 2, \cdots, n\}\rightarrow \{l,r\}$
\begin{align*}
\phi(b_{\chi(1)}, b_{\chi(2)},\cdots, b_{\chi(n)})=&\sum_{\pi\in \P^{(\chi)}(n)}\kappa_{\chi,\pi}(b_{\chi(1)}, b_{\chi(2)},\cdots, b_{\chi(n)})\\
=&\lim_{N\rightarrow\infty}\varphi (S_{\chi(1),N}S_{\chi(2), N}\cdots S_{\chi(n),N})\\
=&\sum_{\pi\in \P^{\chi}(n)}\lim_{N\rightarrow \infty}N^{|\pi|}\kappa_{\chi,\pi}(a_{\chi(1),N,i(1)},a_{\chi(2),N,i(2)},\cdots, a_{\chi(n),N,i(n)}).
\end{align*}
By the definition of bi-free cumulants, the cumulants are determined uniquely by (1.2) (Proposition 5.2 in \cite{MN}). Therefore,
$$\kappa_{\chi,\pi}(b_{\chi(1)}, b_{\chi(2)},\cdots, b_{\chi(n)})=\lim_{N\rightarrow \infty}N^{|\pi|}\kappa_{\chi,\pi}(a_{\chi(1),N,i(1)},a_{\chi(2),N,i(2)},\cdots, a_{\chi(n),N,i(n)}).$$ Especially, by Lemma 2.2,
\begin{align*}
\kappa_\chi(b):=&\kappa_{\chi, 1_n}(b_{\chi(1)}, b_{\chi(2)},\cdots, b_{\chi(n)})\\
=&\lim_{N\rightarrow \infty}N\kappa_{\chi, 1_n}(a_{\chi(1),N,1}, a_{\chi(2),N,1},\cdots, a_{\chi(n),N,1})\\
=&\lim_{N\rightarrow \infty}N\varphi(a_{\chi(1),N,1} a_{\chi(2),N,1}\cdots a_{\chi(n),N,1}).
\end{align*}
\end{proof}

Without any essential difficulties, we can generalize the above limit theorem to the multidimensional case.

For each $N\in \mathbb{N}$ and $1\le m\le N$, let $a_{N,m}=(a_{l,N,m}^{(i)})_{i\in I}, (a_{r,N,m}^{(j)})_{j\in J})$ be a two-faced family, and $\{a_{N,1}, a_{N,2}, \cdots, a_{N,N}\}$ be a bi-free sequence of identically distributed two-faced families of random variables in a non-commutative probability $(\A,\varphi)$, where $I$ and $J$ are disjoint index sets. It follows that, for $n\ge 1$, $\chi:\{1, 2, \cdots, n\}\rightarrow \{l,r\}$, and $\alpha:\{1, 2, \cdots, n\}\rightarrow I\bigsqcup J$, such that $\chi^{-1}(l)=\alpha^{-1}(I)$,  the moments
$$\varphi (a_{\chi(1), N, m}^{(\alpha(1))}a_{\chi(2), N, m}^{(\alpha(2))}\cdots a_{\chi(n), N, m}^{(\alpha(n))}), 1\le m\le N,$$ are independent of $m$. Let $S_{l, N}^{(i)}=\sum_{m=1}^Na_{l,N,m}^{(i)}$, for $i\in I$, $S_{r, N}^{(j)}=\sum_{m=1}^Na_{r,N,m}^{(j)}$, for $j\in J$,  and $S_N=((S_{l,N}^{(i)})_{i\in I}, (S_{r,N}^{(j)})_{j\in J})$ be the two-faced family of random variables in $(\A,\varphi)$. Then we have the following limit theorem.
\begin{Corollary}The following two statements are equivalent.
\begin{enumerate}
\item There is a two-faced family $b=((b_{l,i})_{i\in I}, (b_{r,j})_{j\in J})$ in a non-commutative probability space $(\B, \phi)$ such that
$$S_N\stackrel{distr}{\rightarrow}b, $$ as $N\rightarrow \infty$. 
\item For each $n\ge 1$, $\chi:\{1, 2, \cdots, n\}\rightarrow \{l,r\}$, and $\alpha:\{1, 2, \cdots, n\}\rightarrow I\bigsqcup J$, such that $\chi^{-1}(l)=\alpha^{-1}(I)$,  the limit
$$\lim_{N\rightarrow \infty}N\varphi (a_{\chi(1), N, 1}^{(\alpha(1))}a_{\chi(2), N, 1}^{(\alpha(2))}\cdots a_{\chi(n), N, 1}^{(\alpha(n))})$$ exists.
\end{enumerate}
If the existence of the limits is assumed, then we have
$$\kappa_{\chi, 1_n}(b_{\chi(1), \alpha(1)},b_{\chi(2), \alpha(2)}, \cdots, b_{\chi(n), \alpha(n)})=\lim_{N\rightarrow \infty}N\varphi (a_{\chi(1), N, 1}^{(\alpha(1))}a_{\chi(2), N, 1}^{(\alpha(2))}\cdots a_{\chi(n), N, 1}^{(\alpha(n))}).$$
\end{Corollary}

As an application of the above limit theorems, we can prove Voiculescu's bi-free central limit theorem (7.9 in \cite{Vo1}) in a special (but most popular) case. First let's recall Voiculescu's centered bi-free Gaussian distributions.
\begin{Definition}[7.3 in \cite{Vo1}]
Let $I$ and $J$ be two disjoint index sets. A two-faced family $z=((z_i)_{i\in I}, (z_j)_{j\in J})$ in a non-commutative probability space $(\A,\varphi)$ has a bi-free centered Gaussian distribution if its cumulants satisfy $\kappa_{\alpha, 1_n}(z)=0$, for all $\alpha:\{1, 2, \cdots, n\}\rightarrow I\bigsqcup J$, and $n\in \mathbb{N}, n\ne 2$.
\end{Definition}
Let's use our limit theorems to prove the bi-free central limit theorem.
\begin{Proposition}[7.9 in \cite{Vo1}] Let $z^{(n)}=((z^{(n)}_i)_{i\in I}, (z_j^{(n)})_{j\in J})$, $n\in \mathbb{N}$,  be a bi-free sequence of identically distributed two-faced families in $(\A,\varphi)$ such that $\varphi(z_k^{(n)})=0,\forall k\in I\bigsqcup J$ and $n\in \mathbb{N}$. Let $S_{N,k}=\frac{1}{\sqrt{N}}\sum_{n=1}^Nz_k^{(n)}$, for $k\in I\bigsqcup J, N\in \mathbb{N}$, and $S_N=((S_{N,i})_{i\in I}, (S_{N,j})_{j\in J})$. Then $$S_N\stackrel{distr}{\rightarrow}b,$$ where $b=((b_i)_{i\in I}, (b_j)_{j\in J})$ is a two-faced family in a non-commutative probability space $(\B, \phi)$ having a bi-free centered Gaussian distribution such that $\kappa(b_kb_l)=\varphi(z_k^{(n)}z_l^{(n)}), \forall k, l\in I\bigsqcup J$, and $n\in \mathbb{N}$.
\end{Proposition}
\begin{proof} By Corollary 2.4, we need to show that $\lim_{N\rightarrow \infty}N\varphi(\frac{z^{(m)}_{k(1)}}{\sqrt{N}}\frac{z^{(m)}_{k(2)}}{\sqrt{N}}\cdots \frac{z^{(m)}_{k(n)}}{\sqrt{N}})$ exists, for all $k:\{1,2, \cdots, n \}\rightarrow I\bigsqcup J, n\ge 1$, and $m\in \mathbb{N}$. In fact,
\begin{multline*}
\kappa_{k, 1_n}(b)=\lim_{N\rightarrow \infty}N\varphi(\frac{z^{(m)}_{k(1)}}{\sqrt{N}}\frac{z^{(m)}_{k(2)}}{\sqrt{N}}\cdots \frac{z^{(m)}_{k(n)}}{\sqrt{N}})\\
=\lim_{N\rightarrow \infty}\frac{N}{N^{n/2}}\varphi(z^{(m)}_{k(1)}z^{(m)}_{k(2)}\cdots z^{(m)}_{k(n)})=\delta_{n,2}\varphi (z^{(m)}_{k(1)}\cdots z^{(m)}_{k(n)}).
\end{multline*}
\end{proof}
\section{Bi-free infinitely divisible distributions}
 The goal of this section is to define and study {\sl bi-free infinitely divisible distributions} in a more general setting than that in \cite{GHM}: we do not require that the random variable in the left face commute with that in the right face of a two-faced pair of random variables. First let's give the definition.
 \begin{Definition} A two faced pair $a=(a_l, a_r)$ of self-adjoint operators in a $C^*$-probability space $(\A,\varphi)$ has a bi-free infinitely divisible distribution if for each $N\in \mathbb{N}$, there is a bi-free sequence of $N$ identically distributed two-faced pairs $\{(a_{l,N,i}, a_{r,N.i}): i=1, 2, \cdots, N\}$ of self-adjoint operators  in a $C^*$-probability space $(\A_N, \varphi_N)$ such that $$S_N:=(S_{l,N}, S_{r,N}):=(\sum_{i=1}^Na_{l,N,i}, \sum_{i=1}^Na_{r,N,i})$$ in $(\A_N, \varphi_N)$ has a distribution same as that of $a=(a_l, a_r)$ in $(\A, \varphi)$.
 \end{Definition}
 \begin{Lemma}For each $N\in \mathbb{N}$, let $\H_N=\underbrace{\H\oplus\cdots \oplus \H}_{N \text{ copies of }\H}$. For $f, g\in \H, T\in B(\H)$, let $$\widehat{f}=\frac{f\oplus f\oplus \cdots \oplus f}{\sqrt{N}}\in \H_N, \widehat{g}=\frac{g\oplus g\oplus \cdots \oplus g}{\sqrt{N}}\in \H_N, \widehat{T}=T\oplus T\oplus \cdots \oplus T\in B(\H_N).$$
 Then $\widehat{\Delta}:=\{l(\widehat{f}), l(\widehat{g})^*, r(\widehat{f}), r(\widehat{g})^*, \Lambda_l(\widehat{T}), \Lambda_r(\widehat{T}): f, g\in \H, T\in B(\H)\}$ has a distribution in $(B(\F(\H_N)), \tau_{\H_N})$ same as that  of $\Delta:=\{l(f), r(f), l(g)^*, r(g)^*, \Lambda_l(T), \Lambda_r(T):f, g\in \H, T\in B(\H)\}$  in $(B(\F(\H),\tau_\H)$.
 \end{Lemma}
\begin{proof}
For any $b\in \Delta$, $\tau_\H(b)=0=\tau_{\H_N}(\widehat{b})$. For $b_0, b_1,\cdots b_n, b_{n+1}\in \Delta, n\ge 0$, $\tau_\H(b_0\cdots b_{n+1})\ne 0$ implies that $b_0=\chi(g)^*$, $b_{n+1}=\chi(f)$, $f, g\in \H, \chi\in \{l,r\}$, and $b_1\cdots b_n f=f_0\in \H$. Then $$\tau_\H(b_0\cdots b_{n+1})=\langle b_1\cdots b_n f,g\rangle.$$ Very similarly, $\tau_{\H_N}(\widehat{b}_0\cdots \widehat{b}_{n+1})\ne 0$ implies that $\tau_{\H_N}(\widehat{b}_0\cdots \widehat{b}_{n+1})=\langle \widehat{b}_1\cdots \widehat{b}_n \widehat{f},\widehat{g}\rangle.$

Since $b_1\cdots b_nf\in \H$, there are $k$ operators $\{\chi(i)(f_i): f_i\in \H, \chi(i)\in \{l,r\}, i=1, 2, \cdots, k\}$ and $k$ operators $\{\chi(i)(f_i)^*: f_i\in \H, \chi(i)\in \{l,r\}, i=1, 2, \cdots, k\}$ among $\{b_1, b_2, \cdots, b_n\}$, $2k\le n$.
Let's prove $$\langle b_1\cdots b_n f,g\rangle=\langle \widehat{b}_1\cdots \widehat{b}_n \widehat{f},\widehat{g}\rangle. \eqno (3.1)$$
by induction in $k$.

When $k=0$, it is obvious that (3.1) holds true, because $b_1, \cdots, b_n$ are $\Lambda_\chi (T)$, $T\in B(\H), \chi \in \{l,r\}$.
When $k=1$, after performing actions of $\Lambda_\chi (T), T\in B(\H)$, $\chi\in \{l,r\}$, we can assume that  there is no $\Lambda_\chi(T)$ in $\{b_1, \cdots, b_n\}$. Therefore, $n=2$.

 Case I. $b_1=l(f_1)^*,  b_2=l(f_2)$. or $b_1=r(f_1)^*, b_2=r(f_2)$. In this case, we have $$\langle b_1b_2f,g\rangle=\langle f_2,f_1\rangle\langle f,g\rangle=\langle \widehat{f}_2,\widehat{f}_1\rangle\langle \widehat{f},\widehat{g}\rangle=\langle \widehat{b}_1\widehat{b}_2\widehat{f},\widehat{g}\rangle.$$
 Case II. $b_1=l(f_1)^*, b_2=r(f_2) $.   $$\langle b_1 b_2f,g\rangle=\langle f,f_1\rangle \langle f_2,g\rangle=\langle \widehat{f},\widehat{f}_1\rangle \langle \widehat{f_2},\widehat{g}\rangle=\langle \widehat{b}_1  \widehat{b}_2\widehat{f}, \widehat{g}\rangle.$$
 Case III. $b_1=r(f_1)^*, b_2=l(f_2) $. The discussion is same as that in case II.

Suppose $(3.1)$ holds true for $k< K$. Now we consider the case that  $k=K$. After performing actions of $\Lambda_\chi(T)$, $T\in B(\H), \chi\in \{\l,r\}$, we can  assume that $b_1,\cdots,  b_n $ are creation or annihilation operators. Choose the largest index $i$ such that  $b_i=l(f_i)^*$ or $r(f_i)^*$, and $b_{i+1}=l(f_{i+1})$ or $r(f_{i+1})$. Then $b_j=\chi(f_j), j=i+2, \cdots, n, \chi\in \{l,r\}$.

Case I. $b_i=l(f_i)^*, b_{i+1}=l(f_{i+1})$. Since $l(\xi)r(\zeta)=r(\zeta)l(\xi)$, for all $\xi, \zeta\in \H$, we can write $b_{i+1}\cdots b_n=l(f_{i+1})l(\xi_1)\cdots l(\xi_p)r(\zeta_q)\cdots r(\zeta_1)$, $p+q=n-i-1$. Then,  by the inductive hypothesis, we have
\begin{multline*}
\langle b_1\cdots b_nf,g\rangle=\langle f_{i+1}, f_i\rangle \langle b_1\cdots b_{i-1} b_{i+2}\cdots b_nf,g\rangle\\
=\langle \widehat{f}_{i+1}, \widehat{f}_i\rangle \langle \widehat{b}_1\cdots \widehat{b}_{i-1} \widehat{b}_{i+2}\cdots \widehat{b}_n\widehat{f},\widehat{g}\rangle=\langle \widehat{b}_{1}\cdots \widehat{b}_n\widehat{f},\widehat{g}\rangle.
\end{multline*}

Case II. $b_i=r(f_i)^*, b_{i+1}=l(f_{i+1})$, and $q>0$.
   By inductive hypothesis, we have
\begin{align*}
\langle b_1\cdots b_nf,g\rangle=&\langle b_1\cdots b_{i-1}r(f_i)^*l(f_{i+1})l(\xi_1)\cdots l(\xi_p) r(\zeta_q)r(\zeta_{q-1})\cdots r(\zeta_1)f,g\rangle\\
=&\langle b_1\cdots b_{i-1}l(f_{i+1})l(\xi_1)\cdots l(\xi_p)r(f_i)^* r(\zeta_q)r(\zeta_{q-1})\cdots r(\zeta_1)f,g\rangle\\
=&\langle \zeta_q,f_i\rangle \langle b_1\cdots b_{i-1}l(f_{i+1})l(\xi_1)\cdots l(\xi_p)r(\zeta_{q-1})\cdots r(\zeta_1)f,g\rangle\\
=&\langle \widehat{\zeta}_q,\widehat{f}_i\rangle \langle \widehat{b}_1\cdots \widehat{b}_{i-1}l(\widehat{f}_{i+1}) l(\widehat{\xi}_1)\cdots l(\widehat{\xi}_p)r(\widehat{\zeta}_{q-1})\cdots r(\widehat{\zeta_1})\widehat{f},\widehat{g}\rangle\\
=&\langle \widehat{b}_1\cdots \widehat{b}_{i-1}l(\widehat{f}_{i+1}) l(\widehat{\xi}_1)\cdots l(\widehat{\xi}_p)r(\widehat{f}_i)^*r(\widehat{\zeta}_q)r(\widehat{\zeta}_{q-1})\cdots r(\widehat{\zeta_1})\widehat{f},\widehat{g}\rangle\\
=&\langle \widehat{b}_1\cdots \widehat{b}_{i-1}r(\widehat{f}_i)^*l(\widehat{f}_{i+1})l(\widehat{\xi}_1)\cdots l(\widehat{\xi}_p)r(\widehat{\zeta}_{q})\cdots r(\widehat{\zeta_1})\widehat{f},\widehat{g}\rangle=\langle \widehat{b}_{1}\cdots \widehat{b}_n\widehat{f},\widehat{g}\rangle.
\end{align*}

Case III. $b_i=r(f_i)^*, b_{i+1}=l(f_{i+1})$, and $q=0$.
\begin{align*}
\langle b_1\cdots b_nf,g\rangle=&\langle b_1\cdots b_{i-1}r(f_i)^*l(f_{i+1})l(\xi_1)\cdots l(\xi_p) f,g\rangle\\
=&\langle b_1\cdots b_{i-1}l(f_{i+1})l(\xi_1)\cdots l(\xi_{p-1})[l(\xi_p)r(f_i)^* f],g\rangle\\
=&\langle f,f_i\rangle \langle b_1\cdots b_{i-1}l(f_{i+1})l(\xi_1)\cdots l(\xi_{p-1})\xi_p,g\rangle\\
=&\langle \widehat{f},\widehat{f}_i\rangle \langle \widehat{b}_1\cdots \widehat{b}_{i-1}l(\widehat{f}_{i+1}) l(\widehat{\xi}_1)\cdots l(\widehat{\xi}_{p-1})\widehat{\xi_p},\widehat{g}\rangle\\
=&\langle \widehat{b}_1\cdots \widehat{b}_{i-1}l(\widehat{f}_{i+1})l(\widehat{\xi}_1)\cdots l(\widehat{\xi}_{p-1})[l(\widehat{\xi}_p)r(\widehat{f}_i)^* \widehat{f}],\widehat{g}\rangle\\
=&\langle \widehat{b}_1\cdots \widehat{b}_{i-1}r(\widehat{f}_i)^*l(\widehat{f}_{i+1})l(\widehat{\xi}_1)\cdots l(\widehat{\xi}_p)\widehat{f},\widehat{g}\rangle=\langle \widehat{b}_{1}\cdots \widehat{b}_n\widehat{f},\widehat{g}\rangle.
\end{align*}

Case IV. $b_i=r(f_i)^*, b_{i+1}=r(f_{i+1})$. The proof is same as that for Case I.

Case V. $b_i=l(f_i)^*, b_{i+1}=r(f_{i+1})$, and $p=0$. The proof is same as that for Case III.

Case VI.  $b_i=l(f_i)^*, b_{i+1}=r(f_{i+1})$, and $p>0$. The proof is same as that for Case II.
\end{proof}

The following corollary provides a particular type of bi-free infinitely divisible distributions.
 \begin{Corollary} For $f, g\in \H, \lambda_1, \lambda_2 \in \mathbb{C}$, and $T_1, T_2\in B(\H)$,  we have the following conclusions.
\begin{enumerate}
\item Let  $b=((b_{l,1}, b_{l,2}, b_{l,3}), (b_{r,1}, b_{r,2}, b_{r,3}))=((l(f), l(g)^*, \Lambda_l(T_1)), (r(f), r(g)^*, \Lambda_r(T_2)))$ and $B$ the set of all operators in $b$. Then the bi-free cumulants of elements in $B$ have the following form. Let $\chi$ be a map from $\{1, 2, \cdots, n\}$ into $\{l,r\}$, for $n\in \mathbb{N}$. Then  $\kappa_\chi(b_i)=0, \forall b_i\in B$. For $n= 2$, the cumulant $\kappa_{\chi}(b_1,b_2)\ne 0$ only if $b_1=l(g)^*, \text{or } r(g)^*$, and $b_2=l(f), \text{or } r(f)$. In this case, $$\kappa_\chi(b_1, b_2)=\langle f,g\rangle.$$ For $n>2$, $\kappa_\chi(b_1\cdots b_n)\ne 0$ only if $b_1=l(g)^*, \text{or } r(g)^*$, $b_n=l(f), \text{or } r(f)$, and $b_2, b_3, \cdots, b_{n-1}$ are $\Lambda_{h}(T_i)$, $i=1, 2$, and $h\in \{l,r\}$. In this case, $$\kappa_\chi(b_1, b_2, \cdots, b_n)=\langle b_2\cdots b_{n-1}f,g\rangle.$$
\item Let $T_1, T_2 \in \mathcal{B}(\H)$ be self adjoint, and $\lambda_1, \lambda_2 \in \mathbb{R}$. Then the two-faced pair $a=(a_l, a_r):=(l(f)+l(f)^*+\Lambda_l(T_1)+\lambda_1 1, r(g)+r(g)^*+\Lambda_r(T_2)+\lambda_2 1)$ has a bi-free infinitely divisible distribution.
\end{enumerate}
\end{Corollary}
\begin{proof}
(1).  For $N\in \mathbb{N}$, by Lemma 3.2, $b$ has a distribution same as $$\widehat{b}:=((l(\widehat{f}), l(\widehat{g})^*, \Lambda_l(\widehat{T}_1)), (r(\widehat{f}), r(\widehat{g})^*, \Lambda_r(\widehat{T}_2))).$$ Furthermore, $\widehat{b}$ is the bi-free sum of $\widehat{b}_1+\cdots \widehat{b}_N$ , where $$\widehat{b}_i=((l(\widehat{f}_i), l(\widehat{g}_i)^*, \Lambda_l((\widehat{T}_1)_i), (r(\widehat{f}_i), r(\widehat{g}_i)^*, \Lambda_r((\widehat{T}_2)_i)), \widehat{f}_i=\frac{0\oplus \cdots\oplus 0\oplus f\oplus 0\oplus\cdots \oplus 0}{\sqrt{N}},$$ the $i$-th component of $\widehat{f}$. All summands have the joint distribution of $$((\frac{l(f)}{\sqrt{N}}, \frac{l(g)^*}{\sqrt{N}}, \Lambda_l(T_1)), (\frac{r(f)}{\sqrt{N}}, \frac{r(g)^*}{\sqrt{N}}, \Lambda_r(T_2))).$$
For $b_i\in B$, let $b_{N,i}$ be the corresponding element in $$B_N:=\{\frac{l(f)}{\sqrt{N}}, \frac{l(g)^*}{\sqrt{N}}, \Lambda_l(T_1), \frac{r(f)}{\sqrt{N}}, \frac{r(g)^*}{\sqrt{N}}, \Lambda_r(T_2)\}.$$ For $b_{1}, b_{2}, \cdots, b_{n} \in B$, we have
$\lim_{N\rightarrow \infty}N\varphi_N(b_{N,1}\cdots b_{N,n})\ne 0$ implies that $b_{N,n}=l(f)/\sqrt{N}$ or $b_{N,n}=r(f)/\sqrt{N}$, and $b_{N,1}=l(g)^*/\sqrt{N}$ or $b_{N,1}=r(g)^*/\sqrt{N}$, and others $b_{N,2}, \cdots b_{N,n-1}$ are $\Lambda_{h}(T_i)$, $i=1, 2, h\in \{l,r\}$. In this case, by Corollary 2.4, we have
$$\kappa_\chi(b_1,\cdots, b_n)=\langle b_2b_3\cdots b_{n-1}f,g\rangle. $$

(2). Without loss of generality, we can assume $\lambda_1=\lambda_2=0$. For any $N\in \mathbb{N}$, by Lemma 3.2, $a$ and $\widehat{a}=(l(\widehat{f})+l(\widehat{f})^*+\Lambda_l(\widehat{T_1}), r(\widehat{g})+r(\widehat{g})^*+\Lambda_r(\widehat{T_2}))$ have the same distribution. Moreover, let $$\widehat{a}_i=(l(\widehat{f}_i)+l(\widehat{f}_i)^*+\Lambda_l(\widehat{T_1}_i), r(\widehat{g}_i)+l(\widehat{g}_i)^*+\Lambda_r(\widehat{T_2}_i)),$$ where $\widehat{f}_i, \widehat{g}_i, \widehat{T}_i$ are the i-th summands of the direct sum vectors $\widehat{f}, \widehat{g}$, and operator $\widehat{T}$, respectively. Then by Proposition 1.1,
 $\widehat{a}_1, \widehat{a}_2, \cdots \widehat{a}_N$ are bi-free. It is obvious that $\widehat{a}=\widehat{a}_1+\cdots +\widehat{a}_N$. Hence, $a$ has a bi-free infinitely divisible distribution.
\end{proof}

The following result is a corollary of Proposition 6.4.1 in \cite{CNS2}.
\begin{Lemma} Let $a_{1}, a_2, \cdots a_n$ be $n$ random variables in a non-commutative probability space $(\A,\varphi)$, and $n\ge 2$. If  $a_i=1$,  for some $i$,  $1\le i\le n$, then for every $\chi:\{1, 2, \cdots, n\}\rightarrow \{l,r\}$, $\pi\in \P^{(\chi)}(n)$ and $\{1\}$ is not a block of $\pi$, $\kappa_{\chi,\pi}(a_1, \cdots, a_n)=0$.
\end{Lemma}

\begin{Definition} Let $(\A,\varphi)$ be a non-commutative probability space, $\Upsilon:=\{\kappa_{\chi}:\chi:\{1,2,\cdots, n\}\rightarrow \{l,r\}, n\ge 1\}$ be the set of all bi-free cumulant polynomials of $(\A, \varphi)$. We say that the cumulant set of a two-faced pair $a=(a_l, a_r)$ is conditionally non-negative definite if for every sequence $\chi_i:\{1, 2, \cdots, i\}\rightarrow \{l,r\}, i=1, 2, \cdots, k$, in $\Upsilon$, and $\alpha_1, \alpha_2, \cdots, \alpha_k\in \mathbb{C}$, $$\sum_{n,m=1}^k\alpha_n\overline{\alpha_m}\kappa_{\chi_n\sqcup\chi_m}(a)\ge 0,$$
where $\kappa_{\chi_{n}\sqcup \chi_{m}}(a)=\kappa_{\chi_n\sqcup\chi_m}(a_{\chi_n(1)},\cdots, a_{\chi_{n}(n)}, a_{\chi_{m}(m)}, \cdots, a_{\chi_m(1)})$, $$\chi_n\sqcup \chi_m:\{1, 2, \cdots, n+m\}\rightarrow \{l,r\},$$
\begin{equation*}\chi_n\sqcup \chi_m(i)=
\begin{cases} \chi_n(i), & \text{if } 1\le i\le n;\\
\chi_m(m+n-i+1), & \text{if } n<i\le n+m.
\end{cases}
\end{equation*}
\end{Definition}
\begin{Definition} Let $a=(a_l,a_r)$ be a pair of two self-adjoint operators in a $C^*$-probability space $(\A,\varphi)$. We say $\Upsilon:=\{\kappa_{\chi}(a):\chi:\{1,2,\cdots, n\}\rightarrow \{l,r\}, n\ge 1\}$ is conditionally bounded if for $\lambda\in \{l,r\}$, there exists a positive number $L$ such that $$\sum_{n, m=1}^k\alpha_n\overline{\alpha_m}\kappa_{(\chi_n\cup\lambda)\sqcup(\chi_m\cup\lambda)}(a_{\chi_n(1)}, \cdots, a_{\chi_n(n)}, a_{\lambda}, a_{\lambda}, a_{\chi_m(m)}, \cdots, a_{\chi_m(1)})\le L\sum_{n, m=1}^k\alpha_n\overline{\alpha_m}\kappa_{\chi_n\sqcup\chi_m}(a),$$ $\forall \chi_n:\{1, 2, \cdots, n\}\rightarrow \{l,r\}, \alpha_n\in \mathbb{C}, n=1, 2, \cdots,k,  k\ge 1$.
\end{Definition}
\begin{Theorem} Let $a=(a_l, a_r)$ be a two-faced pair of self-adjoint operators in a $C^*$-probability space $(\A, \varphi)$. The following statements are equivalent.
\begin{enumerate}
\item $a$ has a bi-free infinitely divisible distribution.
\item the bi-free cumulant set $\Upsilon (a):=\{\kappa_\chi(a): \chi: \{1, 2, \cdots, n\}\rightarrow \{l,r\}, n\ge 1\}$ is conditionally non-negative definite and conditionally bounded.
\item $a$ is the limit in distribution of a sequence of triangular arrays in the limit theorem (Theorem 2.3): For each $N\in \mathbb{N}$, there is a bi-free family $\{(a_{l,N,i,}, a_{r,N,i}): i=1, 2, \cdots, N \}$ of identically distributed two-faced pairs of self-adjoint operators in a $C^*$-probability space $(\A_N, \varphi_N)$ such that $$\kappa_\chi(a)=\lim_{N\rightarrow \infty}N\varphi_N(a_{\chi(1), N,i}a_{\chi(2), N, i}\cdots a_{\chi(n), N, i}), \forall \chi:\{1, 2, \cdots, n\}\rightarrow \{l,r\}, k\ge 1,$$ where $1\le i\le N$.
\end{enumerate}
\end{Theorem}
\begin{proof}
$(3)\Rightarrow (2)$. For $k\in \mathbb{N}$, $\alpha_i\in \mathbb{C}, i=1, 2, \cdots, k$, and a sequence $\{\kappa_{\chi_i}\in \Upsilon (a): i=1, 2, \cdots, k \}$, by Theorem 2.3, we have
\begin{align*}
&\sum_{n,m=1}^k\alpha_n\overline{\alpha_m}\kappa_{\chi_n\sqcup\chi_m}(a)\\
=&\lim_{N\rightarrow \infty}N\sum_{m,n=1}^k\alpha_n\overline{\alpha_m}\varphi_N(a_{\chi_n(1), N,i}\cdots a_{\chi_n(n), N, i}a_{\chi_m(m), N, i}\cdots a_{\chi_m(1), N, i})\\
=&\lim_{N\rightarrow\infty}N\varphi_N((\sum_{n=1}^k\alpha_n a_{\chi_{n}(1), N, i}\cdots a_{\chi_n(n), N, i})(\sum_{m=1}^k\alpha_m a_{\chi_{m}(1), N, i}\cdots a_{\chi_m(m),  N, i})^*)\ge 0.
\end{align*}
Now we show that $\Upsilon(a)$ is conditionally bounded. For $\alpha_n\in \mathbb{C}$, $a_{\chi_n}=a_{\chi_n(1)}a_{\chi_n(2)}\cdots a_{\chi_n(n)}$, $\chi_n:\{1, 2, \cdots, n\}\rightarrow \{l,r\}, n=1, 2, \cdots, k, k\ge 1$,  and $\lambda\in \{l,r\}$, by the previous  argument, we have
\begin{align*}
&\sum_{n,m=1}^k\alpha_n\overline{\alpha_m}\kappa_{(\chi_n\cup\lambda) \sqcup(\chi_m\cup\lambda)}(a)\\
=&\lim_{N\rightarrow\infty}N\varphi_N((\sum_{n=1}^k\alpha_n a_{\chi_{n}(1), N, i}\cdots a_{\chi_n(n), N, i})a_\lambda^2(\sum_{m=1}^k\alpha_m a_{\chi_{m}(1), N, i}\cdots a_{\chi_m(m),  N, i})^*)\\
\le& \|a_\lambda^2\|\lim_{N\rightarrow\infty}N\varphi_N((\sum_{n=1}^k\alpha_n a_{\chi_{n}(1), N, i}\cdots a_{\chi_n(n), N, i})(\sum_{m=1}^k\alpha_m a_{\chi_{m}(1), N, i}\cdots a_{\chi_m(m),  N, i})^*)\\
=&L\sum_{n,m=1}^k\alpha_n\overline{\alpha_m}\kappa_{\chi_n\sqcup\chi_m}(a).
\end{align*}
We get the desired result with $L=\|a_\lambda^2\|$.

$(2)\Rightarrow (1)$. Let $\mathbb{C}\langle X_l, X_r\rangle$ be the set of all polynomials in two (non-commutative) variables $X_l$ and $X_r$ without constant terms. Let $$X_\chi:=X_{\chi(1)}X_{\chi(2)}\cdots X_{\chi(n)}, \forall \chi:\{1, 2, \cdots, n \}\rightarrow \{l,r\}, n\ge 1.$$
By (2), we can define an inner product on $\mathbb{C}\langle X_l, X_r\rangle$ by a sesquilinear extension  of $$\langle X_{\chi_n}, X_{\chi_m}\rangle=\kappa_{\chi_m\sqcup\chi_n}(a), \forall \chi_i:\{1, 2, \cdots, i\}\rightarrow \{l, r\}, i=m,n, m, n\ge 1.$$ We, thus, get a Hilbert space $\H$ after dividing out the kernel and completion. After identifying $\mathbb{C}\langle X_l, X_r\rangle$ with its image in $\H$, we may treat elements of $\mathbb{C}\langle X_l, X_r\rangle$ as vectors in $\H$ and operators on $\H$. Consider the $C^*$-probability space $B(\F(\H), \tau_\H)$ and the operators $$b=(b_l, b_r)=(l(X_l)+l(X_l)^*+\Lambda_l(X_l)+\kappa(a_l)1,r(X_r)+r(X_r)^*+\Lambda_l(X_r)+\kappa(a_r)1),$$ where $X_\lambda, \lambda\in\{l,r\}$, is  the {\sl right} multiplication operator of $X_\lambda$ on $\H$, that is, $$X_\lambda(X_{\chi(1)}X_{\chi(2)}\cdots X_{\chi(n)})=X_{\chi(1)}\cdots X_{\chi(n)}X_\lambda,$$ for $X_{\chi(1)}\cdots X_{\chi(n)}\in \mathbb{C}\langle X_l,X_r\rangle\subset \H$.
By (2), $X_\lambda$ is a bounded operator on $\H$, therefore, $\Lambda(X_\lambda)\in \B(\F(\H))$.

Now we show that  $\Lambda(X_\lambda)$ is self adjoint on $\F(\H)$.
For $\chi:\{1, 2, \cdots, n\}\rightarrow \{l,r\}, \delta:\{1, 2, \cdots, m\}\rightarrow \{l,r\}$, we have
\begin{align*}
&\langle X_{\chi(1)}\cdots X_{\chi(n)}, \Lambda(X_\lambda)^*X_{\delta(1)}\cdots X_{\delta(m)}\rangle\\
=&\langle \Lambda(X_\lambda)X_{\chi(1)}\cdots X_{\chi(n)}, X_{\delta(1)}\cdots X_{\delta(m)}\rangle\\
=&\langle X_{\chi(1)}\cdots X_{\chi(n)}X_{\lambda}, X_{\delta(1)}\cdots X_{\delta(m)}\rangle\\
=&\kappa_{\delta\sqcup(\chi\sqcup\lambda)}(a_{\delta(1)}, \cdots, a_{\delta(m)}, a_{\lambda}, a_{\chi(n)},\cdots, a_{\chi(1)})\\
=&\langle X_{\chi(1)}\cdots X_{\chi(n)}, X_{\delta(1)}\cdots X_{\delta(m)}X_\lambda\rangle\\
=&\langle X_{\chi(1)}\cdots X_{\chi(n)}, \Lambda(X_\lambda)X_{\delta(1)}\cdots X_{\delta(m)}\rangle
\end{align*}
It implies that $\Lambda(X_\lambda)=\Lambda(X_\lambda)^*$. Thus, $b_l$ and $b_r$ are self adjoint operators in $\B(\F(\H))$.

With the same method, we can prove that in decomposition $\widehat{a}=\widehat{a}_1+\cdots +\widehat{a}_N$ in Corollary 3.3 (2),  each $\widehat{a}_i$ is a pair of two self adjoint operators in a $C^*$-probability space  for $i=1, 2, \cdots, N, N\in \mathbb{N}$. Hence, by Corollary 3.3, $b=(b_l, b_r)$ has a bi-free infinitely divisible distribution in sense of Definition 3.1.

Finally, we show that $a$'s  distribution in $(\A,\varphi)$ is same as that of $b$ in $(B(\F(\H)), \tau_\H)$.

It is obvious that $\kappa(b_\chi)=\kappa(a_\chi)$, $\chi\in \{l,r\}$.

For $n\ge 2$, let
\begin{multline*}
\widetilde{b}=(b_l, b_r)=((b_{l,1}+b_{l,2}+b_{l,3}), (b_{r,1}, b_{r,2}, b_{r,3}))\\
:=((l(X_l)+l(X_l)^*+\Lambda_l(X_l)), (r(X_r)+r(X_r)^*+\Lambda_r(X_r))).
 \end{multline*}
 By Lemma 3.4, $b$ and $\widetilde{b}$ have the same bi-free cumulants $\kappa_\chi$, for any map $\chi:\{1, 2, \cdots, n\}\rightarrow \{l,r\}, n\ge 2$. For $n=2$, by Corollary 3.3, we have $$\kappa_\chi(b_l, b_r)=\sum_{i,j=1}^3\kappa_\chi(b_{l,i}, b_{r,j})=
 \langle X_r, X_l\rangle= \kappa_\chi(a_l, a_r)$$ and, similarly,  $\kappa(b_r, b_l)=\kappa(a_r, a_l)$. For $n>2$, $\chi:\{1, 2, \cdots, n\}\rightarrow \{l,r\}$, we have
 \begin{align*}
 \kappa_\chi(b_{\chi(1)}, b_{\chi(2)}, \cdots, b_{\chi(n)})&=\sum_{\alpha(1), \alpha(2), \cdots, \alpha(n)=1}^3\kappa_\chi(b_{\chi(1), \alpha(1)}, \cdots, b_{\chi(n), \alpha(n)})\\
 =&\langle b_{\chi(2), 3}\cdots b_{\chi(n-1),3}X_{\chi(n)}, X_{\chi(1)}\rangle\\
 =&\langle \Lambda_{\chi(2)}(X_{\chi(2)})\cdots\Lambda_{\chi(n-1)}(X_{\chi(n-1)})X_{\chi(n)}, X_{\chi(1)}\rangle\\
 =&\langle X_{\chi(n)}X_{\chi(n-1)}\cdots X_{\chi(2)}, X_{\chi(1)}\rangle\\
 =&\kappa_\chi (a_{\chi(1)}, a_{\chi(2)}, \cdots, a_{\chi(n)}).
 \end{align*}
 Hence,  $a$ has a bi-free infinitely divisible distribution.

 $(1)\Rightarrow (3)$ is obvious. For each $N\in \mathbb{N}$, choose  a bi-free sequence $$\{(a_{l, N,1},a_{r, N,1}), (a_{l,N,2},a_{r,N,2}) \cdots, (a_{l,N,N},a_{r,N,N})\}$$ of identically distributed two-faced pairs of self-adjoint operators in a $C^*$-probability space $(\A_N, \varphi_N)$ such that $S_N=(\sum_{i=1}^Na_{l,N,i}, \sum_{i=1}^Na_{r,N,i})$ has a distribution same as that of $a$. It is trivial to see that  $S_N$ converges in distribution to $a$, as $N\rightarrow \infty$.
\end{proof}
\section{Bi-free Levy Processes}
In this section, we investigate the  relation between bi-free Levy processes and bi-free infinitely divisible distributions. Let's give the definition of bi-free Levy processes first.
\begin{Definition}
A family $\{a_t=(a_{l,t}, a_{r,t}): t\ge 0\}$ of two-faced pairs of self-adjoint operators in a $C^*$-probability space $(\A, \varphi)$ is called a bi-free Levy process if it satisfies the following conditions.
\begin{enumerate}
\item $a_0=(0,0)$.
\item If $0\le t_1<t_2<\cdots < t_n<\infty$, then $ a_{t_2}-a_{t_1}, \cdots, a_{t_n}-a_{t_{n-1}} $ are bi-free, where $a_t-a_s=(a_{l,t}-a_{l,s}, a_{r,t}-a_{r,s})$.
\item For $0<s<t$, the distribution of $a_t-a_s$ depends only on $t-s$.
\item The distribution $\mu_t$ of $a_t$ converges to $0$, as $t\rightarrow 0+$.
\end{enumerate}
\end{Definition}
\begin{Theorem}Let $a=(a_l, a_r)$ and $a_t=(a_{l,t}, a_{r,t})$ be two-faced pairs of self-adjoint operators in a $C^*$-probability space $(\A,\varphi)$.
\begin{enumerate}
\item Let  $\{a_t=(a_{l,t}, a_{r,t}):t\ge 0\}$ be a bi-free Levy process. Then $a_1$ has a bi-free infinitely divisible distribution.
\item If $a=(a_l, a_r)$ has a bi-free infinitely divisible distribution, then there is a bi-free Levy process $\{b_t=(b_{l,t}, b_{r,t}): t\ge 0\}$ in a $C^*$-probability space $(\B, \phi)$ such that $a$ and $b_1$ have the same distribution.
\end{enumerate}
\end{Theorem}
\begin{proof}For $0<s, t$, $a_{t+s}=(a_{t+s}-a_{s})+a_s$. By the definition of bi-free Levy processes, $(a_{t+s}-a_{s})$ and $a_s$ are bi-free, and $a_{t+s}-a_s$ and $a_t$ have the same distribution.  Generally, for any $n\in \mathbb{N}$, $$a_1=a_{1/n}+(a_{1}-a_{1/n})=\cdots=a_{1/n}+\sum_{i=1}^{n-1}(a_{(i+1)/n}-a_{i/n}).$$ Hence, $a_1$ has a bi-free infinitely divisible distribution.

Let $a=(a_l, a_r)$ has a bi-free infinitely divisible distribution. By The proof of Theorem 3.7, we can choose $a$ as $$a=(l(X_l)+l(X_l)^*+\Lambda_l(X_l)+\kappa(a_l)1,r(X_r)+r(X_r)^*+\Lambda_l(X_r)+\kappa(a_r)1)$$ on the $C^*$-probability space $(B(\F(\H)), \tau_\H)$, where $\H$ is the Hilbert space obtained from the polynomial set $\mathbb{C}\langle X_l,X_r\rangle$ by a special sesquilinear form defined in the proof of Theorem 3.7. The following construction is adapted from the proof of Theorem 4.2 in \cite{GHM} (originally, from \cite{GSS}). Define a new Hilbert space $\K=L^2(\mathbb{R}_+, dx)\otimes \H$, where $\mathbb{R}_+=[0,\infty)$. For a Borel set $I\subset \mathbb{R}_+$, let $\chi_I$ be the characteristic function of $I$ in $L^2(\mathbb{R}_+, dx)$. $M_I$ be the multiplication operator of $\chi_I$ on $L^2(\mathbb{R}_+, dx)$.  Define $$X_{l,t}=\chi_{[0, t)}\otimes X_l, X_{r, t}=\chi_{[0,t)}\otimes X_r, A_{l,t}=M_{[0,t)}\otimes M(X_l), A_{r,t}=M_{[0,t)}\otimes M(X_r),$$ where $M(X\chi):\H\rightarrow \H$ is the {\sl right} multiplication operator of $X_\chi $ on $\H$, $\chi\in \{l,r\}$.
Then consider the operators $$b_t=(l(X_{l,t})+l(X_{l,t})^*+\Lambda_l({A_{l,t}})+t\kappa(a_l), r(X_{r,t})+r(X_{r,t})^*+\Lambda_r({A_{r,t}})+t\kappa(a_r)), t> 0, b_0=(0,0)$$ on the $C^*$-probability space $B(\F(\K), \tau_\K)$.

For $0<s<t$, define $X_{\alpha, s,t}=\chi_{[s,t)}\otimes X_\alpha, A_{\alpha, s,t}=M_{[s,t)}\otimes M(X_\alpha), \alpha\in \{\l,r\}$. Then
\begin{multline*}
b_t-b_s:=(c_l, c_r):=(c_{l,1}+c_{l,2}+c_{l,3}+(t-s)\kappa(a_l), c_{r,1}+c_{r,2}+c_{r,3}+(t-s)\kappa(a_r))\\
=(l(X_{l,s,t})+l(X_{l,s, t})^*+\Lambda_l({A_{l,s, t}})+(t-s)\kappa(a_l), r(X_{r,s,t})+r(X_{r,s,t})^*+\Lambda_r({A_{r,s,t}})+(t-s)\kappa(a_r)).
\end{multline*}

Let $0<t_0<t_1<\cdots <t_n$. Define $$\K_0=L^2([0,t_0), dx)\otimes\H, \K_i=L^2([t_{i-1}, t_i),dx)\otimes \H, i=1, 2, \cdots, n.$$ Then $\K=\bigoplus_{i=0}^n\K_i$. Let $\B_i$ and $\C_i$ be the unital $C^*$-algebras generated by $\{l(f):f\in \K_i\}\cup \{\Lambda_l(T): T\in B(\K), T\K_i\subset \K_i, T|_{\K\ominus\K_i}=0\}$ and $\{r(f):f\in \K_i\}\cup \{\Lambda_r(T): T\in B(\K), T\K_i\subset \K_i, T|_{\K\ominus\K_i}=0\}$, respectively. By Proposition 1.1,  $\{(\B_i, \C_i):0\le i\le n\}$ is bi-free in $(B(\F(\K)), \tau_\K)$. Since  $X_{\alpha,t_0}=\chi_{[0,t_0)}\otimes X_\alpha\in \K_0$, and $A_{\alpha,t_0}=M_{[0,t_0)}\otimes M(X_\alpha):\K_0\rightarrow \K_0$, and $A_{\alpha,t_0}|_{\K\ominus\K_0}=0$, for $\alpha\in\{l,r\}$, we have $b_{t_0}\in (\B_0, \C_0)$. Very similarly,  $b_{t_1}-b_{t_0}\in (\B_1, \C_1), \cdots, b_{t_n}-b_{t_{n-1}}\in(B_n, \C_n)$.  It implies that $b_{t_0}, b_{t_1}-b_{t_0}, \cdots, b_n-b_{n-1}$ are bi-free.

It is obvious that $\kappa(c_\chi)=(t-s)\kappa(a_\chi), \chi\in \{l,r\}$. For $n\ge 2$, by Corollary 3.3, $\kappa_\chi(c_{\chi(1)},\cdots, c_{\chi(n)})\ne 0$ if and only if $\kappa_\chi(a_{\chi(1)}, \cdots, a_{\chi(n)})\ne 0$. In this case, we have
\begin{align*}
\kappa_\chi(c_{\chi(1)},\cdots, c_{\chi(n)})=&\langle c_{\chi(2),3}\cdots c_{\chi(n-1), 3}X_{\chi(n), s,t}, X_{\chi(1), s,t}\rangle\\
=&\kappa_\chi(a_{\chi(1)},a_{\chi(2)},\cdots, a_{\chi(n)})\langle\chi_{[s,t)}, \chi_{[s,t)}\rangle_{L^2(\mathbb{R}_+,dx)}\\
=&\kappa_\chi(a_{\chi(1)},a_{\chi(2)},\cdots, a_{\chi(n)})(t-s).
\end{align*}
The above discussion also shows that $\kappa_\chi(b_t)=t\kappa_\chi(a), \forall \chi:\{1, 2, \cdots, n\}\rightarrow \{l,r\}, n\ge 1$. Thus, $\mu_t\rightarrow 0$, as $t\rightarrow 0+$.
It follows that $\{b_t:t\ge 0\}$ is a bi-free Levy process, and  $\kappa_\chi(b_1)=\kappa_\chi(a)$, for all $\chi:\{1, 2, \cdots, n\}\rightarrow \{l,r\}, n\ge 1$.
\end{proof}

\noindent

\end{document}